\RequirePackage{fix-cm}
\documentclass[smallextended]{svjour3}       
\smartqed  
\usepackage{graphicx}
\usepackage{amssymb}
\usepackage{amsfonts}
\usepackage{amsmath}
\begin{document}

\title{Trace and extension operators for fractional Sobolev spaces with variable exponent
}

\titlerunning{Trace and extension operators  }        

\author{Azeddine Baalal \and Mohamed Berghout}


\institute{ Azeddine Baalal \at
	Department of Mathematics, Faculty of Sciences A\"{\i}n Chock, University of Hassan II,\\
	B.P. 5366 M\^{a}arif, Casablanca, Morocco. \\
	\email{azeddine.baalal@univh2c.ma }
	\and
	Mohamed Berghout \at
	  Department of Mathematics, Faculty of Sciences A\"{\i}n Chock, University of Hassan II,\\
      B.P. 5366 M\^{a}arif, Casablanca, Morocco. \\
    \email{moh.berghout@gmail.com }          
}

\date{2017}

\maketitle

\begin{abstract}
Assume that $\Omega \subset 
\mathbb{R}
^{n}$ is a bounded domain. We show that under certain regularity assumptions
on $\Omega $ there exists a linear extension operator from the space $%
\mathcal{W}^{s,p\left( .,.\right) }\left( \Omega \right) $ to $\mathcal{W}%
^{s,p\left( .,.\right) }\left( 
\mathbb{R}
^{n}\right) $. As an application we study complemented subspaces in $%
\mathcal{W}^{s,p\left( .,.\right) }$ via the trace operator.
\keywords{extension operator \and trace operator \and fractional variable
exponent Sobolev spaces \and complemented subspace problem.}
\end{abstract}

\section{Introduction}
\label{intro}
Sobolev spaces are very interesting mathematical structures in their own
right, but their principal significative lies in the central role they, and
their numerous generalizations and applications. So it is necessary to
develop a reasonable abstract body related to these spaces. In particular
the problem of how to extend Sobolev functions was recognized early in the
development of the theory of the Sobolev spaces. In this direction many
people were interested in determining the exention of the Sobolev functions
we mention in particular the works of Sobolev (\cite{sobo1}, \cite{sobo2}),
of Deny and Lions (\cite{DL}), of Gagliardo (\cite{Gaglia}) and of the
authors in (\cite{Slo}, \cite{As}, \cite{Liz},\cite{Ste}). Calder\'{o}n(\cite%
{Cal}), Stein(\cite{Ste67}) and Jones (\cite{Jon})\ studied the problem of
extension in $\mathcal{W}^{k,p}\left( \Omega \right) $ for $k\in 
\mathbb{N}
$ and $1<p<\infty $. For the fractional Sobolev Spaces $\mathcal{W}%
^{s,p}\left( \Omega \right) $ where $0<s<1$ we cite (\cite{ED}). For
variable exponent Sobolev functions $\mathcal{W}^{k,p\left( .\right) }\left(
\Omega \right) $ there exists a extension operator $\mathcal{E}$ from $%
\mathcal{W}^{k,p\left( .\right) }\left( \Omega \right) $ to $\mathcal{W}%
^{k,p\left( .\right) }\left( 
\mathbb{R}
^{n}\right) $ (see \cite{LD}). All these previous results are holds under
certain crucial regularity assumptions on the domain $\Omega $. The
objective of this paper is to study the problem of extension in the
fractional Sobolev spaces with variable exponent $\mathcal{W}^{s,p\left(
.,.\right) }\left( \Omega \right) $ and its relation with the trace
operator. Precisely we show that the existence of an extension operator
Implies that:
\begin{itemize}
\item the trace operator is surjective,

\item the kernel of the trace operator is complemented in $\mathcal{W}%
^{s,p\left( .,.\right) }\left( 
\mathbb{R}
^{n}\right) $.
\end{itemize}

First motivation of this paper is the following extension problem: if $%
\Omega \subset 
\mathbb{R}
^{n}$ is a domain with regular geometry of $\partial \Omega $ and $u\in 
\mathcal{W}^{s,p\left( .,.\right) }\left( \Omega \right) $, can be extend $u$
to the whole of $%
\mathbb{R}
^{n}$?. Second motivation of studying these spaces is that solutions of
partial differential equations belong naturally to Sobolev spaces, and in
particular for the study of partial differential equations related to the $%
p\left( .\right) -$Laplacian operator, it can be much more fruitful to study
weaker forms of equations in Sobolev spaces $\mathcal{W}^{k;p\left( .\right)
}$, where notions of smoothness are relaxed to require only the existence of
weak derivatives. For the nonlocal fractional $p\left( .\right) -$Laplacian
operator, we constrain functions in a fractional variable exponent Sobolev
spaces $\mathcal{W}^{s,p\left( .,.\right) }\left( \Omega \right) $ where $%
s\in \left( 0,1\right) $. Variable exponent spaces $L^{p\left( .\right)
}\left( \Omega \right) $ and $\mathcal{W}^{k,p\left( .\right) }\left( \Omega
\right) $ with $k\in 
\mathbb{N}
$ have been studied in many papers; for surveys see (\cite{LD},\cite{DCruze},%
\cite{meshki}).

This paper consists of four sections. After an introduction, we give a
notation and preliminaries used throughout this paper. In section three
under certain regularity assumptions on the domain $\Omega $\ we prove the
existence of the linear extension operator $%
\begin{array}{ccc}
\mathcal{E}: & \mathcal{W}^{s;p\left( .,.\right) }\left( \Omega \right)
\rightarrow & \mathcal{W}^{s;p\left( .,.\right) }\left( 
\mathbb{R}
^{n}\right)%
\end{array}%
$ such that for all $u\in \mathcal{W}^{s;p\left( .,.\right) }\left( \Omega
\right) $, $\mathcal{E}u\mid _{\Omega }=u$. Firstly we prove the existence
when the function $u$ is identically zero in a neighborhood of the boundary $%
\partial \Omega $, after that we prove the extension theorem for any domain
satisfying certain regularity assumptions. In the last section, we define a
trace operator in fractional Sobolev spaces with variable exponent and we
prove that his kernel is complemented in $\mathcal{W}^{s,p\left( .,.\right)
}\left( 
\mathbb{R}
^{n}\right) $.
\section{Notation and preliminaries}
\label{sec:1}
Throughout of this paper we will use the following notation: ${\mathbb{R}}%
^{n}$ is the real Euclidean $n-$space. $\Omega $ is a open subset of ${%
\mathbb{R}}^{n}$. For $E\subset {\mathbb{R}}^{n}$ measurable we denote by $%
\left\vert E\right\vert $ the Lebesgue measure of $E$. If $A\subset {\mathbb{%
R}}^{n}$ we denote by $\partial A$ the topological boundary of $A$, by $%
\overline{A}$ the closure of $A$ and $d=diam(A)=sup\{|x-y|:x,y\in A\}$
denote the diameter of a set $A$. $C$ will denote a constant which may
change even in single string of an estimate. By ${supp}f$ $\ $we
denote the support of the function $f$.

We say that a closed subspace $Y$ of a Banach space $X$ is complemented if
there is another closed subspace $Z$ of $X$ such that $X=Y\oplus Z$. That
is, $Y\cap Z=\{0\}$ and every element $x\in X$ can be written as $x=y+z$,
with $y\in Y$ and $z\in Z$. By $Y^{\perp }$ we denote the orthogonal of $Y$.
For an operator $L$ we denote by $KerL$ the kernel of $L$.

We say that an open set $D\subset {\mathbb{R}}^{n}$ is of class $C^{1}$ if
for every $x\in \partial D$ there exist a neighborhood $\emph{U}$ of $\ x\in 
\mathbb{R}
^{n}$ and a bijective map $%
\begin{array}{ccc}
H: & Q\mathbb{\rightarrow } & \emph{U}%
\end{array}%
$ such that:

\begin{equation*}
H\in C^{1}\left( \overline{Q}\right) ,H^{-1}\in C^{1}\left( \overline{\emph{U%
}}\right) ,H\left( Q_{+}\right) =\emph{U}\cap Q,\text{ and }H\left(
Q_{0}\right) =\emph{U}\cap \partial D\text{,}
\end{equation*}
where for given $x\in {\mathbb{R}}^{n}$ write $x=\left( x^{^{\prime
}},x_{n}\right) $ with $x^{^{\prime }}\in $ ${\mathbb{R}}^{n-1}$ and $%
x^{^{\prime }}=\left( x_{1},......,x_{n-1}\right) $,

$\left\vert x^{^{\prime }}\right\vert =\left(
\sum_{i=1}^{n-1}x_{i}^{2}\right) ^{\frac{1}{2}},$ $\ \ \ \ \ \ \ $

$\ {\mathbb{R}}_{+}^{n}:=\left\{ x=\left( x^{^{\prime }},x_{n}\right)
;x_{n}>0\right\} ,$

$Q:=\left\{ x=\left( x^{^{\prime }},x_{n}\right) ;\left\vert x^{^{\prime
}}\right\vert <1\text{ and }\left\vert x_{n}\right\vert <1\text{ }\right\} ,$

$Q_{+}:=Q\mathbb{\cap }{\mathbb{R}}_{+}^{n},$

$Q_{0}:=\left\{ x=\left( x^{^{\prime }},0\right) ;\left\vert x^{^{\prime
}}\right\vert <1\right\} $. The map $H$ is called a local chart.

For an open set $D\subset {\mathbb{R}}^{n}$, we denote by $\mathcal{C}%
^{k}\left( D\right) $ the set of functions with $k$-th derivative is
continuous for $k$ positive integer, $\mathcal{C}^{\infty }\left( D\right)
=\cap _{k\geqslant 1}\mathcal{C}^{k}\left( D\right) $ and by $\mathcal{C}%
_{0}^{\infty }\left( D\right) $ the set of all functions in $\mathcal{C}%
^{\infty }\left( D\right) $ with compact support.

We will next introduce variable exponent Lebesgue spaces and fractional
exponent Sobolev spaces; note that we nevertheless use the standard
definitions of the spaces $L^{p}\left( D\right) $ and $\mathcal{W}%
^{s,p}\left( D\right) $ in the fixed exponent case with open $D\subset 
\mathbb{R}
^{n}$.

Let $\left( \mathcal{A},\sum ,\mu \right) $ be a $\sigma -$finite complete
mesure space and $%
\begin{array}{ccc}
p: & \mathcal{A}\rightarrow & \left[ 1,\infty \right)%
\end{array}%
$ be a $\mu -$mesurable function (called \textit{the variable exponent} on $%
\mathcal{A}$). We define $p^{+}:=p_{\mathcal{A}}^{+}:={esssup}
_{x\in \mathcal{A}}$ $p\left( x\right) $, $p^{-}:=p_{\mathcal{A}%
}^{-}:={essinf}_{x\in \mathcal{A}}$ $p\left( x\right) $, $%
\alpha :=\min \left\{ 1,\min \left\{ p^{+},p^{-}\right\} \right\} $ and $%
\beta :=\max \left\{ 1,\max \left\{ p^{+},p^{-}\right\} \right\} $ . We say
that $p$ is a bounded variable exponent if $p^{+}<\infty $. By $\mathcal{M}%
\left( \mathcal{A},\mu \right) $ we denote the space of all measurable $\mu
- $function $%
\begin{array}{ccc}
u: & \mathcal{A}\rightarrow & 
\mathbb{R}%
\end{array}%
$. The variable exponent \textit{Lesbegue} $L^{p\left( .\right) }\left( 
\mathcal{A},\mu \right) $ is defined as follows: 
\begin{equation*}
{\small L}^{p\left( .\right) }\left( \mathcal{A},\mu \right) {\small :=}%
\left\{ 
\begin{array}{c}
u\in \mathcal{M}\left( \mathcal{A},\mu \right) :\rho _{p\left( .\right)
}\left( \lambda u\right) =\int_{\mathcal{A}}\left\vert \lambda u\right\vert
^{p\left( x\right) }dx<\infty \text{,} \\ 
\text{ for some }\lambda >0%
\end{array}%
\right\} \text{,}
\end{equation*}%
where the function $%
\begin{array}{ccc}
\rho _{p\left( .\right) }: & L^{p\left( .\right) }\left( \mathcal{A},\mu
\right) \rightarrow & \left[ 0,\infty \right)%
\end{array}%
$ is called \textit{the modular }of the space $L^{p\left( .\right) }\left( 
\mathcal{A},\mu \right) $. We define a norm, the so-called \textit{Luxemburg
norm}, on this space by the formula 
\begin{equation*}
\left\Vert u\right\Vert _{p\left( .\right) }=\inf \left\{ \lambda >0:\rho
_{p\left( .\right) }\left( \lambda u\right) \leq 1\right\} ,
\end{equation*}%
which makes $L^{p\left( .\right) }\left( \mathcal{A},\mu \right) $ a Banach
space. In the classical case when $\mu $ is the $n$-\textit{Lesbegue }%
mesure, $\Omega $ is a smooth bounded domain of $%
\mathbb{R}
^{n}$, $\sum $ is the $\sigma $-algebra of $\mu $-measurable subsets of $%
\Omega $ and $p\in \mathcal{M}\left( \Omega ,\mu \right) $ is bounded
variable exponent, we simply denote $L^{p\left( .\right) }\left( \Omega ,\mu
\right) $ by $L^{p\left( .\right) }\left( \Omega \right) $. We denote by $%
L^{p^{^{\prime }}\left( .\right) }\left( \Omega \right) $ the conjugate
space of $L^{p\left( .\right) }\left( \Omega \right) $ where $\frac{1}{%
p\left( x\right) }+\frac{1}{p^{^{\prime }}\left( x\right) }=1$, we have the
following theorem where his proof is in (\cite{LD}):

\begin{theorem}
(H\"{o}lder's inequality ) For any $u\in L^{p\left( .\right) }\left( \Omega
\right) $ and $v\in L^{p^{^{\prime }}\left( .\right) }\left( \Omega \right) $%
, we have 
\begin{equation*}
\left\vert \int_{\Omega }uvdx\right\vert \leq 2\left\Vert u\right\Vert
_{L^{p\left( .\right) }\left( \Omega \right) }\left\Vert v\right\Vert
_{L^{p^{^{\prime }}\left( .\right) }\left( \Omega \right) }.
\end{equation*}
\end{theorem}

The variable exponent \textit{Sobolev }space $\mathcal{W}^{1,p\left(
.\right) }\left( \Omega \right) $ is the space of all measurable function $%
\begin{array}{ccc}
u: & \Omega \rightarrow & 
\mathbb{R}%
\end{array}%
$ such that $u$ and the absolute value of distributional gradient $\nabla
u=\left( \partial _{1}u,....,\partial _{n}u\right) $ is in $L^{p\left(
.\right) }\left( \Omega \right) $. The norm $\left\Vert u\right\Vert _{%
\mathcal{W}^{1,p\left( .\right) }\left( \Omega \right) }=\left\Vert
u\right\Vert _{L^{p\left( .\right) }\left( \Omega \right) }+\left\Vert
\nabla u\right\Vert _{L^{p\left( .\right) }\left( \Omega \right) }$ makes $%
\mathcal{W}^{1,p\left( .\right) }\left( \Omega \right) $ a Banach space.

From now let $\Omega $ be a fixed smooth bounded domain in ${\mathbb{R}}^{n}$
and \ $0<s<1$. Let $p$ be a bounded \ continuous variable exponent in $%
\overline{\Omega }\times \overline{\Omega }$, $q$ be a bounded \ continuous
variable exponent in $\overline{\Omega }$ and $p^{\ast }\left( x\right) :=%
\frac{np\left( x,x\right) }{n-sp\left( x,x\right) }$ the fractional critical
variable Sobolev exponent with $p^{\ast }\left( x\right) >p^{+}q\left(
x\right) $ \ and $q\left( x\right) >$ $p\left( x,x\right) $ for $x\in 
\overline{\Omega }$. We want to define the fractional Sobolev spaces with
variable exponent $\mathcal{W}^{s,p\left( .,.\right) }\left( \Omega \right) $%
, for this we extend the definition given in (\cite{ED}) to the case of
variable exponent. So we define the fractional Sobolev spaces with variable
exponent $\mathcal{W}^{s,p\left( .,.\right) }\left( \Omega \right) $ as
follows:%
\begin{equation*}
\mathcal{W}^{s,p\left( .,.\right) }\left( \Omega \right) :=\left\{ u\in
L^{q\left( .\right) }\left( \Omega \right) :\frac{\left\vert u\left(
x\right) -u\left( y\right) \right\vert }{\left\vert x-y\right\vert ^{\frac{n%
}{p\left( x,y\right) }+s}}\in L^{p\left( .,.\right) }\left( \Omega \times
\Omega \right) \right\} \text{;}
\end{equation*}%
i.e, an intermediary Banach space between two Banach spaces $L^{q\left(
.\right) }\left( \Omega \right) $ and $\mathcal{W}^{1,p\left( .\right)
}\left( \Omega \right) $, endowed with the natural norm 
\begin{equation*}
\left\Vert u\right\Vert _{\mathcal{W}^{s,p\left( .,.\right) }\left( \Omega
\right) }:=\left\Vert u\right\Vert _{L^{q\left( .\right) }\left( \Omega
\right) }+\left[ u\right] _{\mathcal{W}^{s,p\left( .,.\right) }\left( \Omega
\right) },
\end{equation*}%
where 
\begin{equation*}
\left[ u\right] _{\mathcal{W}^{s,p\left( .,.\right) }\left( \Omega \right)
}:=\inf \left\{ \lambda >0:\int_{\Omega }\int_{\Omega }\frac{\left\vert
u\left( x\right) -u\left( y\right) \right\vert ^{p\left( x,y\right) }}{%
\lambda ^{p\left( x,y\right) }\left\vert x-y\right\vert ^{n+sp\left(
x,y\right) }}dxdy\leq 1\right\} ,
\end{equation*}%
is the semi norm of $u$.

For general theory of classical Sobolev spaces we refer the reader to (\cite%
{Adams}, \cite{Adams2}, \cite{B.O.T}, \cite{Breziz}, \cite{ED}, \cite{E.H}, 
\cite{V.I.BuREBKOV}) and for the Lesbegue and Sobolev with variable exponent
to (\cite{DCruze},\cite{LD},\cite{meshki}).

\begin{theorem}
\label{injec sobolev} Let $\Omega \subset $ $%
\mathbb{R}
^{n}$ be a smooth bounded domain. Assume that ${\small sp}\left( x,y\right) 
{\small <n}$ for $\left( x,y\right) \in $ $\overline{\Omega }\times 
\overline{\Omega }$, $q\left( x\right) >$ $p\left( x,x\right) $ for $\ x\in $
$\overline{\Omega }$ \ and \ $%
\begin{array}{ccc}
r: & \overline{\Omega } & \rightarrow \left( 1,\infty \right)%
\end{array}%
$ is a continous function such that $p^{\ast }\left( x\right) >r\left(
x\right) \geq r^{-}>1$, for $x\in \overline{\Omega }$. Then the space $%
\mathcal{W}^{s,p\left( .,.\right) }\left( \Omega \right) $ is continously
embedded in $L^{r\left( x\right) }\left( \Omega \right) $ for any $r\in
\left( 1,p^{\ast }\right) $.
\end{theorem}

\begin{proof}
we find it in (\cite{URIEL}).\qed
\end{proof}

\begin{lemma}
\label{Estimat integral}Let $\Omega $ be a smooth bounded domain in $%
\mathbb{R}
^{n}$. Then there exists a suitable positive constant $C$ such that%
\begin{equation*}
\int_{\Omega }\int_{\Omega }\frac{\left\vert u\left( x\right) -u\left(
y\right) \right\vert ^{p\left( x,y\right) }}{\left\vert x-y\right\vert
^{n+sp\left( x,y\right) }}dxdy\leq C\left( \left\Vert u\right\Vert _{%
\mathcal{W}^{s,p\left( .,.\right) }\left( \Omega \right)
}^{p^{+}}+\left\Vert u\right\Vert _{\mathcal{W}^{s,p\left( .,.\right)
}\left( \Omega \right) }^{p^{-}}\right)
\end{equation*}%
for all $u$ $\in $ $\mathcal{W}^{s,p\left( .,.\right) }\left( \Omega \right) 
$.
\end{lemma}

\begin{proof}
Let $u$ $\in \mathcal{W}^{s,p\left( .,.\right) }\left( \Omega \right) $,
then;%
\begin{eqnarray*}
&&\int_{\Omega }\int_{\Omega }\frac{\left\vert u\left( x\right) -u\left(
y\right) \right\vert ^{p\left( x,y\right) }}{\left\vert x-y\right\vert
^{n+sp\left( x,y\right) }}dxdy \\
&\leq &2^{^{p^{+}-1}}\int_{\Omega }\int_{\Omega }\frac{\left\vert u\left(
x\right) \right\vert ^{p\left( x,y\right) }+\left\vert u\left( y\right)
\right\vert ^{p\left( x,y\right) }}{\left\vert x-y\right\vert ^{n+sp\left(
x,y\right) }}dxdy \\
&\leq &2^{^{p^{+}}}\int_{\Omega }\int_{\Omega }\frac{\left\vert u\left(
x\right) \right\vert ^{p\left( x,y\right) }}{\left\vert x-y\right\vert
^{n+sp\left( x,y\right) }}dxdy \\
&\leq &2^{^{p^{+}}}\int_{\Omega }\int_{\Omega \cap \left\{ \left\vert
x-y\right\vert \geq 1\right\} }\frac{\left\vert u\left( x\right) \right\vert
^{p\left( x,y\right) }}{\left\vert x-y\right\vert ^{n+sp\left( x,y\right) }}%
dxdy \\
&&+2^{^{p^{+}}}\int_{\Omega }\int_{\Omega \cap \left\{ \left\vert
x-y\right\vert <1\right\} }\frac{\left\vert u\left( x\right) \right\vert
^{p\left( x,y\right) }}{\left\vert x-y\right\vert ^{n+sp\left( x,y\right) }}%
dxdy \\
&\leq &2^{^{p^{+}}}\int_{\Omega }\left( \int_{\left\{ \left\vert
x-y\right\vert \geq 1\right\} }\frac{dy}{\left\vert x-y\right\vert
^{n+sp\left( x,y\right) }}\right) \left\vert u\left( x\right) \right\vert
^{p\left( x,y\right) }dx \\
&&+2^{^{p^{+}}}\int_{\Omega }\left( \int_{\left\{ \left\vert x-y\right\vert
<1\right\} }\frac{dy}{\left\vert x-y\right\vert ^{n+sp\left( x,y\right) }}%
\right) \left\vert u\left( x\right) \right\vert ^{p\left( x,y\right) }dx \\
&\leq &2^{^{p^{+}}}\int_{\Omega }\left( \int_{\left\{ \left\vert
x-y\right\vert \geq 1\right\} }\frac{dy}{\left\vert x-y\right\vert
^{n+sp\left( x,y\right) }}\right) \left\vert u\left( x\right) \right\vert
^{p\left( x,y\right) }dx \\
&&+2^{^{p^{+}}}\int_{\Omega }\left( \int_{\left\{ \left\vert x-y\right\vert
<1\right\} }\frac{dy}{\left\vert x-y\right\vert ^{n+\left( s-1\right)
p\left( x,y\right) }}\right) \frac{\left\vert u\left( x\right) \right\vert
^{p\left( x,y\right) }}{\left\vert x-y\right\vert ^{p\left( x,y\right) }}dx
\\
&\leq &2^{^{p^{+}}}\int_{\Omega }\left( \int_{\left\{ \left\vert
x-y\right\vert \geq 1\right\} }\frac{dy}{\left\vert x-y\right\vert
^{n+sp\left( x,y\right) }}\right) \left( \left\vert u\left( x\right)
\right\vert ^{p^{+}}+\left\vert u\left( x\right) \right\vert ^{p^{-}}\right)
dx+ \\
&&2^{^{p^{+}}}\max \left( d^{^{-p^{+}}},d^{^{-p^{-}}}\right) \int_{\Omega
}\left( \int_{\left\{ \left\vert x-y\right\vert <1\right\} }\frac{dy}{%
\left\vert x-y\right\vert ^{n+\left( s-1\right) p\left( x,y\right) }}\right)
\times \\
&&\left( \left\vert u\left( x\right) \right\vert ^{p^{+}}+\left\vert u\left(
x\right) \right\vert ^{p^{-}}\right) dx \\
&\leq &2C\left( n,s,p^{+}\right) \left\Vert 1\right\Vert _{L^{q^{^{\prime
}}\left( .\right) _{\left( \Omega \right) }}}\left( \left\Vert u\right\Vert
_{L^{^{p^{+}q\left( .\right) }}\left( \Omega \right) }^{p^{+}}+\left\Vert
u\right\Vert _{L^{^{p^{-}q\left( .\right) }}\left( \Omega \right)
}^{p^{-}}\right) \\
&&+2C\left( n,s,p^{+},p^{-},d\right) \left\Vert 1\right\Vert
_{L^{q^{^{\prime }}\left( .\right) _{\left( \Omega \right) }}}\left(
\left\Vert u\right\Vert _{L^{^{p^{+}q\left( .\right) }}\left( \Omega \right)
}^{p^{+}}+\left\Vert u\right\Vert _{L^{^{p^{-}q\left( .\right) }}\left(
\Omega \right) }^{p^{-}}\right) \\
&\leq &2C\left( n,s,p^{+},\left\vert \Omega \right\vert \right) \left(
\left\Vert u\right\Vert _{L^{^{p^{+}q\left( .\right) }}\left( \Omega \right)
}^{p^{+}}+\left\Vert u\right\Vert _{L^{^{p^{-}q\left( .\right) }}\left(
\Omega \right) }^{p^{-}}\right) \\
&&+2C\left( n,s,p^{+},p^{-},d,\left\vert \Omega \right\vert \right) \left(
\left\Vert u\right\Vert _{L^{^{p^{+}q\left( .\right) }}\left( \Omega \right)
}^{p^{+}}+\left\Vert u\right\Vert _{L^{^{p^{-}q\left( .\right) }}\left(
\Omega \right) }^{p^{-}}\right) \\
&\leq &C\left( \left\Vert u\right\Vert _{L^{^{p^{+}q\left( .\right) }}\left(
\Omega \right) }^{p^{+}}+\left\Vert u\right\Vert _{L^{^{p^{-}q\left(
.\right) }}\left( \Omega \right) }^{p^{-}}\right) \\
&\leq &C\left( \left\Vert u\right\Vert _{\mathcal{W}^{s,p\left( .,.\right)
}\left( \Omega \right) }^{p^{+}}+\left\Vert u\right\Vert _{\mathcal{W}%
^{s,p\left( .,.\right) }\left( \Omega \right) }^{p^{-}}\right)
\end{eqnarray*}%
where $C:=\max \left\{ 2C\left( n,s,p^{+},\left\vert \Omega \right\vert
\right) ,2C\left( n,s,p^{+},p^{-},d,\left\vert \Omega \right\vert \right)
\right\} $. Note that this inequality follows in order from the fact that
the kernel $\frac{1}{\left\vert x-y\right\vert ^{n+sp\left( x,y\right) }}$
is summable with respect to $y$ if $|x-y|\geq 1$ since $n+sp\left(
x,y\right) >n$, on the other hand, the kernel $\frac{1}{\left\vert
x-y\right\vert ^{n+\left( s-1\right) p\left( x,y\right) }}$ is summable when 
$|x-y|<1$ since \ $n+\left( s-1\right) p\left( x,y\right) <n$, by using the H%
\"{o}lder's inequality, and finally by using theorem \ref{injec sobolev}. So
for all $u$ $\in $ $\mathcal{W}^{s,p\left( .,.\right) }\left( \Omega \right)
:$ 
\begin{equation*}
\int_{\Omega }\int_{\Omega }\frac{\left\vert u\left( x\right) -u\left(
y\right) \right\vert ^{p\left( x,y\right) }}{\left\vert x-y\right\vert
^{n+sp\left( x,y\right) }}dxdy\leq C\left( \left\Vert u\right\Vert _{%
\mathcal{W}^{s,p\left( .,.\right) }\left( \Omega \right)
}^{p^{+}}+\left\Vert u\right\Vert _{\mathcal{W}^{s,p\left( .,.\right)
}\left( \Omega \right) }^{p^{-}}\right) .
\end{equation*}\qed
\end{proof}

\section{Sobolev Extension Operators}
\label{sec:2}
To study the properties of the fractional sobolev spaces with variable
exponent $\mathcal{W}^{s,p\left( .,.\right) }\left( \Omega \right) $ it is
often preferable to beginning with the case $\Omega =%
\mathbb{R}
^{n}$. It is therefore useful to be able to extend a function $u\in \mathcal{%
W}^{s,p\left( .,.\right) }\left( \Omega \right) $ to a function $\widetilde{u%
}\in \mathcal{W}^{s,p\left( .,.\right) }\left( 
\mathbb{R}
^{n}\right) $.

We start with some preliminary lemmas, in which we will construct the
extension to the whole of $%
\mathbb{R}
^{n}$ of a function $u$ defined on $\Omega $.

\begin{lemma}
\label{u is identically zero in neig of bound}Let $\Omega $ be an open set
in $%
\mathbb{R}
^{n}$, $u$ a function in $\mathcal{W}^{s,p\left( .,.\right) }\left( \Omega
\right) $. If there exists a compact subset $K\subset \Omega $ such that $%
u\equiv 0$ in $\Omega \backslash K$, then the extension function $\widetilde{%
u}$ defined as $\ $%
\begin{equation*}
\widetilde{u}\left( x\right) =\left\{ 
\begin{array}{c}
u\left( x\right) \text{ \ if\ \ \ \ }x\in \Omega ; \\ 
0\text{ \ \ \ \ if\ \ \ \ }x\in 
\mathbb{R}
^{n}\backslash \Omega \text{\ .\ \ \ \ \ \ \ \ }%
\end{array}%
\right.
\end{equation*}%
belongs to $\mathcal{W}^{s,p\left( .,.\right) }\left( 
\mathbb{R}
^{n}\right) $ and 
\begin{equation*}
\left\Vert \widetilde{u}\right\Vert _{W^{s,p\left( .,.\right) }\left( 
\mathbb{R}
^{n}\right) }\leq \left\{ 
\begin{array}{c}
C\left\Vert u\right\Vert _{\mathcal{W}^{s,p\left( .,.\right) }\left( \Omega
\right) }^{\beta }\text{ if }\left\Vert u\right\Vert _{\mathcal{W}%
^{s,p\left( .,.\right) }\left( \Omega \right) }\geq 1\text{,} \\ 
C\left\Vert u\right\Vert _{\mathcal{W}^{s,p\left( .,.\right) }\left( \Omega
\right) }^{\alpha }\text{, if }\left\Vert u\right\Vert _{\mathcal{W}%
^{s,p\left( .,.\right) }\left( \Omega \right) }\leq 1\text{,}%
\end{array}%
\right.
\end{equation*}

where $C$ is a suitable positive constant.
\end{lemma}

\begin{proof}
\bigskip \bigskip By construction we have $\widetilde{u}$ $\in L^{q\left(
.\right) }\left( 
\mathbb{R}
^{n}\right) $ and 
\begin{equation*}
\left\Vert \widetilde{u}\right\Vert _{L^{q\left( .\right) }\left( 
\mathbb{R}
^{n}\right) }=\left\Vert u\right\Vert _{L^{q\left( .\right) }\left( \Omega
\right) }\leq \left\Vert u\right\Vert _{W^{s,p\left( .,.\right) }\left(
\Omega \right) }
\end{equation*}

Hence we show that

\begin{equation*}
\int_{%
\mathbb{R}
^{n}}\int_{%
\mathbb{R}
^{n}}\frac{\left\vert \widetilde{u}\left( x\right) -\widetilde{u}\left(
y\right) \right\vert ^{p\left( x,y\right) }}{\lambda ^{p\left( x,y\right)
}\left\vert x-y\right\vert ^{n+sp\left( x,y\right) }}dxdy<\infty ,\text{ for
some }\lambda >0\text{.}
\end{equation*}

We have

\begin{eqnarray*}
&&\int_{%
\mathbb{R}
^{n}}\int_{%
\mathbb{R}
^{n}}\frac{\left\vert \widetilde{u}\left( x\right) -\widetilde{u}\left(
y\right) \right\vert ^{p\left( x,y\right) }}{\lambda ^{p\left( x,y\right)
}\left\vert x-y\right\vert ^{n+sp\left( x,y\right) }}dxdy \\
&=&\int_{\Omega }\int_{\Omega }\frac{\left\vert u\left( x\right) -u\left(
y\right) \right\vert ^{p\left( x,y\right) }}{\lambda ^{p\left( x,y\right)
}\left\vert x-y\right\vert ^{n+sp\left( x,y\right) }}dxdy \\
&&+2\int_{\Omega }\left( \int_{%
\mathbb{R}
^{n}\backslash \Omega }\frac{\left\vert u\left( x\right) \right\vert
^{p\left( x,y\right) }}{\lambda ^{p\left( x,y\right) }\left\vert
x-y\right\vert ^{n+sp\left( x,y\right) }}dy\right) dx,
\end{eqnarray*}

since $u\in $ $W^{s,p\left( .,.\right) }\left( \Omega \right) $, then 
\begin{equation*}
\int_{\Omega }\int_{\Omega }\frac{\left\vert u\left( x\right) -u\left(
y\right) \right\vert ^{p\left( x,y\right) }}{\lambda ^{p\left( x,y\right)
}\left\vert x-y\right\vert ^{n+sp\left( x,y\right) }}dxdy<\infty \text{.}
\end{equation*}
For all $y\in 
\mathbb{R}
^{n}\backslash K$, 
\begin{eqnarray*}
&&\frac{\left\vert u\left( x\right) \right\vert ^{p\left( x,y\right) }}{%
\lambda ^{p\left( x,y\right) }\left\vert x-y\right\vert ^{n+sp\left(
x,y\right) }} \\
&=&\frac{\chi _{K\left( x\right) }\left\vert u\left( x\right) \right\vert
^{p\left( x,y\right) }}{\lambda ^{p\left( x,y\right) }\left\vert
x-y\right\vert ^{n+sp\left( x,y\right) }} \\
&\leq &\frac{1}{\min \left\{ \lambda ^{p^{-}},\lambda ^{p^{+}}\right\} }%
\sup_{x\in K}\frac{1}{\left\vert x-y\right\vert ^{n+sp\left( x,y\right) }}%
\chi _{K\left( x\right) }\left\vert u\left( x\right) \right\vert ^{p\left(
x,y\right) } \\
&\leq &\frac{1}{\min \left\{ \lambda ^{p^{-}},\lambda ^{p^{+}}\right\} }%
\frac{1}{dist\left( y,\partial K\right) ^{n+sp\left( x,y\right) }}\chi
_{K\left( x\right) }\left\vert u\left( x\right) \right\vert ^{p\left(
x,y\right) } \\
&\leq &\frac{1}{\min \left\{ \lambda ^{p^{-}},\lambda ^{p^{+}}\right\} }%
\frac{\chi _{K\left( x\right) }}{dist\left( y,\partial K\right) ^{n+sp\left(
x,y\right) }}\left( \left\vert u\left( x\right) \right\vert
^{p^{+}}+\left\vert u\left( x\right) \right\vert ^{p^{-}}\right)
\end{eqnarray*}%
so 
\begin{eqnarray*}
&&\int_{\Omega }\left( \int_{%
\mathbb{R}
^{n}\backslash \Omega }\frac{\left\vert u\left( x\right) \right\vert
^{p\left( x,y\right) }}{\lambda ^{p\left( x,y\right) }\left\vert
x-y\right\vert ^{n+sp\left( x,y\right) }}dy\right) dx \\
&\leq &\frac{1}{\min \left\{ \lambda ^{p^{-}},\lambda ^{p^{+}}\right\} }%
\int_{\Omega }\int_{%
\mathbb{R}
^{n}\backslash \Omega }\frac{\chi _{K\left( x\right) }}{dist\left(
y,\partial K\right) ^{n+sp\left( x,y\right) }}\left( \left\vert u\left(
x\right) \right\vert ^{p^{+}}+\left\vert u\left( x\right) \right\vert
^{p^{-}}\right) dxdy,
\end{eqnarray*}%
using the H\"{o}lder inequality, we get 
\begin{eqnarray*}
&&\int_{\Omega }\left( \int_{%
\mathbb{R}
^{n}\backslash \Omega }\frac{\left\vert u\left( x\right) \right\vert
^{p\left( x,y\right) }}{\lambda ^{p\left( x,y\right) }\left\vert
x-y\right\vert ^{n+sp\left( x,y\right) }}dy\right) dx \\
&\leq &\frac{2\left\Vert \chi _{K}\right\Vert _{q^{^{\prime }}\left(
.\right) }}{\min \left\{ \lambda ^{p^{-}},\lambda ^{p^{+}}\right\} }\left(
\int_{%
\mathbb{R}
^{n}\backslash \Omega }\frac{1}{dist\left( y,\partial K\right) ^{n+sp\left(
x,y\right) }}dy\right) \\
&&\times \left( \left\Vert u\right\Vert _{L^{^{p^{+}q\left( .\right)
}}\left( \Omega \right) }^{p^{+}}+\left\Vert u\right\Vert
_{L^{^{p^{-}q\left( .\right) }}\left( \Omega \right) }^{p^{-}}\right) \\
&\leq &\frac{2\max \left\{ 1,\left\vert K\right\vert \right\} }{\min \left\{
\lambda ^{p^{-}},\lambda ^{p^{+}}\right\} }\left( \int_{%
\mathbb{R}
^{n}\backslash \Omega }\frac{1}{dist\left( y,\partial K\right) ^{n+sp\left(
x,y\right) }}dy\right) \\
&&\times \left( \left\Vert u\right\Vert _{L^{^{p^{+}q\left( .\right)
}}\left( \Omega \right) }^{p^{+}}+\left\Vert u\right\Vert
_{L^{^{p^{-}q\left( .\right) }}\left( \Omega \right) }^{p^{-}}\right) ,
\end{eqnarray*}%
by theorem \ref{injec sobolev} we know that there existe a positive constant 
$C$ such that

\begin{equation*}
\left\Vert u\right\Vert _{L^{^{p^{+}q\left( .\right) }}\left( \Omega \right)
}^{p^{+}}\leq C\left\Vert u\right\Vert _{\mathcal{W}^{s,p\left( .,.\right)
}\left( \Omega \right) }^{p^{+}}\text{ \ and \ }\left\Vert u\right\Vert
_{L^{^{p^{-}q\left( .\right) }}\left( \Omega \right) }^{p^{-}}\leq
C\left\Vert u\right\Vert _{\mathcal{W}^{s,p\left( .,.\right) }\left( \Omega
\right) }^{p^{-}}\text{.}
\end{equation*}%
On the other hand 
\begin{equation*}
\left( \int_{%
\mathbb{R}
^{n}\backslash \Omega }\frac{1}{dist\left( y,\partial K\right) ^{n+sp\left(
x,y\right) }}dy\right) <\infty \text{,}
\end{equation*}
since $\ n+sp\left( x,y\right) >n$ and dist$\left( y,\partial K\right) >0$.
Using the lemma \ref{Estimat integral} we get 
\begin{eqnarray*}
&&\int_{%
\mathbb{R}
^{n}}\int_{%
\mathbb{R}
^{n}}\frac{\left\vert \widetilde{u}\left( x\right) -\widetilde{u}\left(
y\right) \right\vert ^{p\left( x,y\right) }}{\lambda ^{p\left( x,y\right)
}\left\vert x-y\right\vert ^{n+sp\left( x,y\right) }}dxdy \\
&\leq &\frac{C}{\min \left\{ \lambda ^{p^{-}},\lambda ^{p^{+}}\right\} }%
\left( \left\Vert u\right\Vert _{\mathcal{W}^{s,p\left( .,.\right) }\left(
\Omega \right) }^{p^{+}}+\left\Vert u\right\Vert _{\mathcal{W}^{s,p\left(
.,.\right) }\left( \Omega \right) }^{p^{-}}\right) \\
&&+\frac{2C\max \left\{ 1,\left\vert K\right\vert \right\} }{\min \left\{
\lambda ^{p^{-}},\lambda ^{p^{+}}\right\} }\left( \left\Vert u\right\Vert _{%
\mathcal{W}^{s,p\left( .,.\right) }\left( \Omega \right)
}^{p^{+}}+\left\Vert u\right\Vert _{\mathcal{W}^{s,p\left( .,.\right)
}\left( \Omega \right) }^{p^{-}}\right) \\
&\leq &\frac{\max \left\{ C,2C\max \left\{ 1,\left\vert K\right\vert
\right\} \right\} }{\min \left\{ \lambda ^{p^{-}},\lambda ^{p^{+}}\right\} }%
\left( \left\Vert u\right\Vert _{\mathcal{W}^{s,p\left( .,.\right) }\left(
\Omega \right) }^{p^{+}}+\left\Vert u\right\Vert _{\mathcal{W}^{s,p\left(
.,.\right) }\left( \Omega \right) }^{p^{-}}\right) \\
&\leq &C\left( \left\Vert u\right\Vert _{\mathcal{W}^{s,p\left( .,.\right)
}\left( \Omega \right) }^{p^{+}}+\left\Vert u\right\Vert _{\mathcal{W}%
^{s,p\left( .,.\right) }\left( \Omega \right) }^{p^{-}}\right) ,
\end{eqnarray*}%
which implies that 
\begin{equation*}
\left[ \widetilde{u}\right] _{^{_{\mathcal{W}^{s,p\left( .,.\right) }\left( 
\mathbb{R}
^{n}\right) }}}\leq C\left( \left\Vert u\right\Vert _{\mathcal{W}^{s,p\left(
.,.\right) }\left( \Omega \right) }^{p^{+}}+\left\Vert u\right\Vert _{%
\mathcal{W}^{s,p\left( .,.\right) }\left( \Omega \right) }^{p^{-}}\right) 
\text{,}
\end{equation*}%
so 
\begin{eqnarray*}
&&\left\Vert \widetilde{u}\right\Vert _{L^{q\left( .\right) }\left( 
\mathbb{R}
^{n}\right) }+\left[ \widetilde{u}\right] _{^{_{\mathcal{W}^{s,p\left(
.,.\right) }\left( \Omega \right) \left( 
\mathbb{R}
^{n}\right) }}} \\
&\leq &\max \left\{ 1,C\right\} \left( \left\Vert u\right\Vert _{\mathcal{W}%
^{s,p\left( .,.\right) }\left( \Omega \right) }+\left\Vert u\right\Vert _{%
\mathcal{W}^{s,p\left( .,.\right) }\left( \Omega \right)
}^{p^{+}}+\left\Vert u\right\Vert _{\mathcal{W}^{s,p\left( .,.\right)
}\left( \Omega \right) }^{p^{-}}\right) \text{,}
\end{eqnarray*}

Consequently 
\begin{equation*}
\left\Vert \widetilde{u}\right\Vert _{\mathcal{W}^{s,p\left( .,.\right)
}\left( 
\mathbb{R}
^{n}\right) }\leq \left\{ 
\begin{array}{c}
C\left\Vert u\right\Vert _{\mathcal{W}^{s,p\left( .,.\right) }\left( \Omega
\right) }^{\beta }\text{ \ if }\left\Vert u\right\Vert _{\mathcal{W}%
^{s,p\left( .,.\right) }\left( \Omega \right) }\geq 1\text{,} \\ 
C\left\Vert u\right\Vert _{\mathcal{W}^{s,p\left( .,.\right) }\left( \Omega
\right) }^{\alpha }\text{ \ if }\left\Vert u\right\Vert _{\mathcal{W}%
^{s,p\left( .,.\right) }\left( \Omega \right) }\leq 1\text{,}%
\end{array}%
\right.
\end{equation*}%
where $C$ is a suitable positive constant.
\end{proof}

\begin{lemma}
\label{extension par reflexion}Let $\Omega $ be an open set in $%
\mathbb{R}
^{n}$, symmetric with respect to the coordinate $x_{n}$, and consider the
sets \ $\Omega _{+}=\{x\in \Omega :x_{n}>0\}$ $\ $and $\Omega _{-}=\{x\in
\Omega :x_{n}\leq 0\}$. Let $u\in \mathcal{W}^{s,p\left( .,.\right) }\left(
\Omega _{+}\right) $. We define the function $\widetilde{u}$ extended by
reflection%
\begin{equation*}
\widetilde{u}=\left\{ 
\begin{array}{c}
u\left( x^{^{\prime }},x_{n}\right) \text{ if }x_{n}\geq 0\text{,} \\ 
u\left( x^{^{\prime }},-x_{n}\right) \text{ if }x_{n}<0\text{.}%
\end{array}%
\right.
\end{equation*}%
Then $\widetilde{u}\in \mathcal{W}^{s,p\left( .,.\right) }\left( \Omega
\right) $.
\end{lemma}

\begin{proof}
By splitting the integrals and changing variable $\widehat{x}=\left(
x^{^{\prime }},-x_{n}\right) $, we get%
\begin{eqnarray*}
\int_{\Omega }\left\vert \widetilde{u}\right\vert ^{q\left( x\right) }dx
&=&\int_{\Omega _{+}}\left\vert \widetilde{u}\right\vert ^{q\left( x\right)
}dx+\int_{\Omega _{-}}\left\vert \widetilde{u}\right\vert ^{q\left( x\right)
}dx \\
&=&\int_{\Omega _{+}}\left\vert u\right\vert ^{q\left( x\right)
}dx+\int_{\Omega _{+}}\left\vert u\left( \widehat{x}^{^{\prime }},\widehat{x}%
_{n}\right) \right\vert ^{q\left( x\right) }d\widehat{x} \\
&=&2\int_{\Omega _{+}}\left\vert u\right\vert ^{q\left( x\right) }dx.
\end{eqnarray*}%
since $u\in W^{s,p\left( .,.\right) }\left( \Omega _{+}\right) $ then $%
\int_{\Omega _{+}}\left\vert \widetilde{u}\right\vert ^{q\left( x\right)
}dx<\infty $, and therefore $\widetilde{u}\in L^{q\left( .\right) }\left(
\Omega \right) $.

Also, we have 
\begin{eqnarray*}
&&\int_{\Omega }\int_{\Omega }\frac{\left\vert \widetilde{u}\left( x\right) -%
\widetilde{u}\left( y\right) \right\vert ^{p\left( x,y\right) }}{\lambda
^{p\left( x,y\right) }\left\vert x-y\right\vert ^{n+sp\left( x,y\right) }}%
dxdy \\
&=&\int_{\Omega _{+}}\int_{\Omega _{+}}\frac{\left\vert u\left( x\right)
-u\left( y\right) \right\vert ^{p\left( x,y\right) }}{\lambda ^{p\left(
x,y\right) }\left\vert x-y\right\vert ^{n+sp\left( x,y\right) }}dxdy \\
&&+2\int_{\Omega _{+}}\int_{%
\mathbb{R}
^{n}\backslash \Omega _{+}}\frac{\left\vert u\left( x\right) -u\left(
y^{^{\prime }},-y_{n}\right) \right\vert ^{p\left( x,y\right) }}{\lambda
^{p\left( x,y\right) }\left\vert x-y\right\vert ^{n+sp\left( x,y\right) }}%
dxdy \\
&&+\int_{%
\mathbb{R}
^{n}\backslash \Omega _{+}}\int_{%
\mathbb{R}
^{n}\backslash \Omega _{+}}\frac{\left\vert u\left( x^{^{\prime
}},-x_{n}\right) -u\left( y^{^{\prime }},-y_{n}\right) \right\vert ^{p\left(
x,y\right) }}{\lambda ^{p\left( x,y\right) }\left\vert x-y\right\vert
^{n+sp\left( x,y\right) }}dxdy \\
&<&\infty \text{.}
\end{eqnarray*}

This concludes the proof.\qed
\end{proof}

\begin{remark}
Note that this lemma \bigskip gives a very simple construction of extension
operators for certain open sets that are not necessarily smooth.
\end{remark}

Now, a truncation lemma near of the boundary $\partial \Omega $.

\begin{lemma}
\label{truncation lemma near the boundary}Let \ $\Omega $ be an open set in $%
\mathbb{R}
^{n}.$ Let us consider $u$ $\in $ $\mathcal{W}^{s,p\left( .,.\right) }\left(
\Omega \right) $ and $\psi \in \mathcal{C}_{0}^{\infty }\left( \Omega
\right) $, $0\leq \psi \leq 1$. Then $\psi u$ $\in $ $\mathcal{W}^{s,p\left(
.,.\right) }\left( \Omega \right) $ and 
\begin{equation*}
\left\Vert \psi u\right\Vert _{\mathcal{W}^{s,p\left( .,.\right) }\left(
\Omega \right) }\leq \left\{ 
\begin{array}{c}
C\left\Vert u\right\Vert _{\mathcal{W}^{s,p\left( .,.\right) }\left( \Omega
\right) }^{\beta }\text{ if }\left\Vert u\right\Vert _{\mathcal{W}%
^{s,p\left( .,.\right) }\left( \Omega \right) }\geq 1\text{,} \\ 
C\left\Vert u\right\Vert _{\mathcal{W}^{s,p\left( .,.\right) }\left( \Omega
\right) }^{\alpha }\text{ if }\left\Vert u\right\Vert _{\mathcal{W}%
^{s,p\left( .,.\right) }\left( \Omega \right) }\leq 1\text{,}%
\end{array}%
\right.
\end{equation*}

where $C$ is a suitable positive constant.
\end{lemma}

\begin{proof}
\bigskip Since $\left\vert \psi \right\vert \leq 1$, it follows that 
\begin{equation*}
\int_{\Omega }\left\vert \frac{\psi u}{\lambda }\right\vert ^{q\left(
x\right) }dx\leq \int_{\Omega }\left\vert \frac{u}{\lambda }\right\vert
^{q\left( x\right) }dx
\end{equation*}%
and consequently:%
\begin{equation*}
\inf \left\{ \lambda >0,\int_{\Omega }\left\vert \frac{\psi u}{\lambda }%
\right\vert ^{q\left( x\right) }dx\leq 1\right\} \leq \inf \left\{ \lambda
>0,\int_{\Omega }\left\vert \frac{u}{\lambda }\right\vert ^{q\left( x\right)
}dx\leq 1\right\}
\end{equation*}%
so 
\begin{equation}
\left\Vert \psi u\right\Vert _{_{L^{^{q\left( .\right) }}\left( \Omega
\right) }}\leq \left\Vert u\right\Vert _{_{L^{^{q\left( .\right) }}\left(
\Omega \right) }}\leq \left\Vert u\right\Vert _{\mathcal{W}^{s,p\left(
.,.\right) }\left( \Omega \right) }\text{.}  \label{2.1}
\end{equation}%
On the other hand 
\begin{eqnarray*}
&&\int_{\Omega }\int_{\Omega }\frac{\left\vert \psi \left( x\right) u\left(
x\right) -\psi \left( y\right) u\left( y\right) \right\vert ^{p\left(
x,y\right) }}{\lambda ^{p\left( x,y\right) }\left\vert x-y\right\vert
^{n+sp\left( x,y\right) }}dxdy \\
&\leq &\int_{\Omega }\int_{\Omega }\frac{\left\vert \psi \left( x\right)
u\left( x\right) -\psi \left( x\right) u\left( y\right) \right\vert
^{p\left( x,y\right) }}{\lambda ^{p\left( x,y\right) }\left\vert
x-y\right\vert ^{n+sp\left( x,y\right) }}dxdy \\
&&+\int_{\Omega }\int_{\Omega }\frac{\left\vert \psi \left( x\right) u\left(
y\right) -\psi \left( y\right) u\left( y\right) \right\vert ^{p\left(
x,y\right) }}{\lambda ^{p\left( x,y\right) }\left\vert x-y\right\vert
^{n+sp\left( x,y\right) }}dxdy \\
&\leq &\frac{1}{\min \left\{ \lambda ^{p^{-}},\lambda ^{p^{+}}\right\} }%
\int_{\Omega }\int_{\Omega }\frac{\left\vert u\left( x\right) -u\left(
y\right) \right\vert ^{p\left( x,y\right) }}{\left\vert x-y\right\vert
^{n+sp\left( x,y\right) }}dxdy \\
&&+\frac{1}{\min \left\{ \lambda ^{p^{-}},\lambda ^{p^{+}}\right\} }%
\int_{\Omega }\int_{\Omega }\frac{\left\vert u\left( y\right) \right\vert
^{p\left( x,y\right) }\left\vert \psi \left( x\right) -\psi \left( y\right)
\right\vert ^{p\left( x,y\right) }}{\left\vert x-y\right\vert ^{n+sp\left(
x,y\right) }}dxdy.
\end{eqnarray*}%
Since $\psi \in \mathcal{C}_{0}^{\infty }\left( \Omega \right) $, we have 
\begin{eqnarray*}
\frac{\left\vert \psi \left( x\right) -\psi \left( y\right) \right\vert
^{p\left( x,y\right) }}{\left\vert x-y\right\vert ^{n+sp\left( x,y\right) }}
&\leq &k^{\max \left( p^{+},p^{-}\right) }\left\vert x-y\right\vert ^{\left(
1-s\right) p\left( x,y\right) -n} \\
&\leq &k^{\max \left( p^{+},p^{-}\right) }\max \left\{ d^{\left( 1-s\right)
p^{+}-n},d^{\left( 1-s\right) p^{-}-n}\right\} ,
\end{eqnarray*}%
where $k$\ denotes the Lipschitz constant of $\psi $. So

\begin{eqnarray*}
&&\int_{\Omega }\int_{\Omega }\frac{\left\vert u\left( y\right) \right\vert
^{p\left( x,y\right) }\left\vert \psi \left( x\right) -\psi \left( y\right)
\right\vert ^{p\left( x,y\right) }}{\left\vert x-y\right\vert ^{n+sp\left(
x,y\right) }}dxdy \\
&\leq &k^{\max \left( p^{+},p^{-}\right) }\max \left\{ d^{\left( 1-s\right)
p^{+}-n},d^{\left( 1-s\right) p^{-}-n}\right\} \int_{\Omega }\int_{\Omega
}\left\vert u\left( y\right) \right\vert ^{p\left( x,y\right) }dxdy \\
&\leq &k^{\max \left( p^{+},p^{-}\right) }\max \left\{ d^{\left( 1-s\right)
p^{+}-n},d^{\left( 1-s\right) p^{-}-n}\right\} \\
&&\times \int_{\Omega }\int_{\Omega }\left( \left\vert u\left( y\right)
\right\vert ^{p^{+}}+\left\vert u\left( y\right) \right\vert ^{p^{-}}\right)
dxdy \\
&\leq &k^{\max \left( p^{+},p^{-}\right) }\max \left\{ d^{\left( 1-s\right)
p^{+}-n},d^{\left( 1-s\right) p^{-}-n}\right\} 2\left\vert \Omega
\right\vert \left\Vert 1\right\Vert _{L^{q^{^{\prime }}\left( .\right)
}\left( \Omega \right) } \\
&&\times \left( \left\Vert u\right\Vert _{L^{p^{+}q\left( .\right) }\left(
\Omega \right) }^{p^{+}}+\left\Vert u\right\Vert _{L^{p^{-}q\left( .\right)
}\left( \Omega \right) }^{p^{-}}\right) \\
&\leq &C\left( k,p^{+},p^{-},d,s,\left\vert \Omega \right\vert \right)
\left( \left\Vert u\right\Vert _{\mathcal{W}^{s,p\left( .,.\right) }\left(
\Omega \right) }^{p^{+}}+\left\Vert u\right\Vert _{\mathcal{W}^{s,p\left(
.,.\right) }\left( \Omega \right) }^{p^{-}}\right) \text{.}
\end{eqnarray*}%
Now combining this last inequality with the lemma (\ref{Estimat integral})
we can find a suitable positive constant $C$ such that 
\begin{eqnarray*}
&&\int_{\Omega }\int_{\Omega }\frac{\left\vert \psi \left( x\right) u\left(
x\right) -\psi \left( y\right) u\left( y\right) \right\vert ^{p\left(
x,y\right) }}{\lambda ^{p\left( x,y\right) }\left\vert x-y\right\vert
^{n+sp\left( x,y\right) }}dxdy \\
&\leq &C\left( \left\Vert u\right\Vert _{\mathcal{W}^{s,p\left( .,.\right)
}\left( \Omega \right) }^{p^{+}}+\left\Vert u\right\Vert _{\mathcal{W}%
^{s,p\left( .,.\right) }\left( \Omega \right) }^{p^{-}}\right) \text{,}
\end{eqnarray*}%
which implies that 
\begin{equation}
\left[ \psi u\right] _{^{_{_{\mathcal{W}^{s,p\left( .,.\right) }\left(
\Omega \right) }}}}\leq C\left( \left\Vert u\right\Vert _{\mathcal{W}%
^{s,p\left( .,.\right) }\left( \Omega \right) }^{p^{+}}+\left\Vert
u\right\Vert _{\mathcal{W}^{s,p\left( .,.\right) }\left( \Omega \right)
}^{p^{-}}\right)  \label{2.2}
\end{equation}

combining (\ref{2.1}) with (\ref{2.2}) we obtain 
\begin{equation*}
\left\Vert \psi u\right\Vert _{\mathcal{W}^{s,p\left( .,.\right) }\left(
\Omega \right) }\leq \left\{ 
\begin{array}{c}
C\left\Vert u\right\Vert _{\mathcal{W}^{s,p\left( .,.\right) }\left( \Omega
\right) }^{\beta }\text{ if }\left\Vert u\right\Vert _{\mathcal{W}%
^{s,p\left( .,.\right) }\left( \Omega \right) }\geq 1\text{,} \\ 
C\left\Vert u\right\Vert _{\mathcal{W}^{s,p\left( .,.\right) }\left( \Omega
\right) }^{\alpha }\text{ if }\left\Vert u\right\Vert _{\mathcal{W}%
^{s,p\left( .,.\right) }\left( \Omega \right) }\leq 1\text{.}%
\end{array}%
\right.
\end{equation*}\qed
\end{proof}

Now, we are ready to state and prove the extension theorem for any domain $%
\Omega $ satisfying certain regularity assumptions.

\begin{theorem}
\label{extension operator}Suppose that $\Omega $ is of classe $C^{1}$ with $%
\partial \Omega $ bounded. Then there exists a linear extension operator 
\begin{equation*}
\mathcal{E}:\mathcal{W}^{s,p\left( .,.\right) }\left( \Omega \right)
\rightarrow \mathcal{W}^{s,p\left( .,.\right) }\left( 
\mathbb{R}
^{n}\right)
\end{equation*}%
such that for all $u\in \mathcal{W}^{s,p\left( .,.\right) }\left( \Omega
\right) $,

\begin{itemize}
\item $\mathcal{E}_{u}\mid \Omega =u$,

\item 
\begin{equation*}
\left\Vert \mathcal{E}u\right\Vert _{\mathcal{W}^{s,p\left( .,.\right)
}\left( 
\mathbb{R}
^{n}\right) }\leq \left\{ 
\begin{array}{c}
C\left\Vert u\right\Vert _{\mathcal{W}^{s,p\left( .,.\right) }\left( \Omega
\right) }^{\beta }\text{ if }\left\Vert u\right\Vert _{\mathcal{W}%
^{s,p\left( .,.\right) }\left( \Omega \right) }\geq 1\text{,} \\ 
C\left\Vert u\right\Vert _{\mathcal{W}^{s,p\left( .,.\right) }\left( \Omega
\right) }^{\alpha }\text{ if }\left\Vert u\right\Vert _{\mathcal{W}%
^{s,p\left( .,.\right) }\left( \Omega \right) }\leq 1\text{,}%
\end{array}%
\right.
\end{equation*}%
where $C$ is a suitable positive constant.
\end{itemize}
\end{theorem}

To show this theorem we need the following lemma

\begin{lemma}
(partition of unity). Let $\Gamma $ be a compact subset of $\ 
\mathbb{R}
^{n}$ and $\emph{U}_{1},.....\emph{U}_{k}$ be a open covering of $\Gamma $.
Then there exist functions $\theta _{0},\theta _{1},....,\theta _{k}\in
C^{\infty }\left( 
\mathbb{R}
^{n}\right) $ such that:

\begin{itemize}
\item $0\leq \theta _{i}\leq 1$, $\forall i=0,1,....,k$ and $%
\sum\limits_{i=0}^{k}\theta _{i}=1$ on $%
\mathbb{R}
^{n}$,

\item supp$\theta _{i}$ is compact, supp$\theta _{i}\subset \emph{U}_{i}$
for all $i=1,2,...k$ and $supp\theta _{0}\subset 
\mathbb{R}
^{n}\backslash \Gamma $.
\end{itemize}

If $\Omega $ is an open bounded set and $\Gamma =\partial \Omega $, then $%
\theta _{0}\mid \Omega \in $ $C_{c}^{\infty }\left( \Omega \right) $.
\end{lemma}

\begin{proof}
This lemma is classical; similar statements can be found, for example, in (%
\cite{Adams2}).\qed
\end{proof}

\begin{proof}
of theorem\bigskip\ \ref{extension operator} \\
We rectify $\partial
\Omega $ by local charts and use a partition of unity. Precisely, since $%
\partial \Omega $ is \ \ \ compact of class $\mathcal{C}^{1}$, we can find a
finite number of balls $B_{i}$ such that: $\partial \Omega \subset
\bigcup\limits_{i=1}^{k}B_{i}$, \ $%
\mathbb{R}
^{n}=\bigcup\limits_{j=1}^{k}B_{i}\cup \left( 
\mathbb{R}
^{n}\backslash \partial \Omega \right) $ and bijective maps: $%
\begin{array}{ccc}
H_{i}: & Q\mathbb{\rightarrow } & B_{i}%
\end{array}%
$ such that: $H_{i}\in C^{1}\left( \overline{Q}\right) $, $H_{i}^{-1}\in
C^{1}\left( \overline{B_{i}}\right) ,$ $H_{i}\left( Q_{+}\right) =B_{i}\cap
\Omega $, and $H_{i}\left( Q_{0}\right) =B_{i}\cap \partial \Omega $. There
exist $k$ smooth functions $\psi _{0},\psi _{1},...,\psi _{k}$ such that $%
supp\psi _{0}\subset 
\mathbb{R}
^{n}\backslash \partial \Omega $, $supp\psi _{i}\subset B_{i}$ for any $i\in
\left\{ 1,2,...,k\right\} $, $0\leq \psi _{i}\leq 1$ for any $i\in \left\{
0,1,2,...,k\right\} $ and $\sum\limits_{i=0}^{k}\psi _{i}=1$.

Given $u\in \mathcal{W}^{s,p\left( .,.\right) }\left( \Omega \right) $, then
:%
\begin{equation*}
u=\sum\limits_{i=0}^{k}\psi _{i}u
\end{equation*}%
by lemma \ref{truncation lemma near the boundary}, we know that $\psi
_{0}u\in \mathcal{W}^{s,p\left( .,.\right) }\left( \Omega \right) $. Since $%
\psi _{0}u\equiv 0$ in a neighborhood of $\partial \Omega $, then by lemma %
\ref{u is identically zero in neig of bound}, we can extend it to the whole
of $%
\mathbb{R}
^{n}$, by setting: 
\begin{equation*}
\widetilde{\psi _{0}u}\left( x\right) =\left\{ 
\begin{array}{c}
\psi _{0}u\left( x\right) \text{ \ if\ \ \ \ }x\in \Omega \text{,} \\ 
0\text{ \ \ \ \ if\ \ \ \ }x\in 
\mathbb{R}
^{n}\backslash \Omega \text{,\ \ \ \ \ \ }%
\end{array}%
\right.
\end{equation*}%
and 
\begin{equation*}
\left\Vert \widetilde{\psi _{0}u}\right\Vert _{\mathcal{W}^{s,p\left(
.,.\right) }\left( \Omega \right) }\leq \left\{ 
\begin{array}{c}
C\left\Vert u\right\Vert _{\mathcal{W}^{s,p\left( .,.\right) }\left( \Omega
\right) }^{\beta }\text{ if }\left\Vert u\right\Vert _{\mathcal{W}%
^{s,p\left( .,.\right) }\left( \Omega \right) }\geq 1\text{,} \\ 
C\left\Vert u\right\Vert _{\mathcal{W}^{s,p\left( .,.\right) }\left( \Omega
\right) }^{\alpha }\text{ if }\left\Vert u\right\Vert _{\mathcal{W}%
^{s,p\left( .,.\right) }\left( \Omega \right) }\leq 1\text{,}%
\end{array}%
\right.
\end{equation*}%
where $C$ is a suitable positive constant.

Now, we extend $u_{i}$, $1\leq i\leq k$\ $\ $where $u_{i}=$ $\psi _{i}u$.
Consider the restriction of $u$ to $B_{i}\cap \Omega $ and transfer this
function to $Q_{+}$ with the help of $H_{i}$.

For any $i\in \left\{ 1,...,k\right\} $, let us consider $u\mid B_{i}\cap
\Omega $ and set: $v_{i}\left( y\right) :=u\left( H_{i}\left( y\right)
\right) $ for all $y\in $ $Q_{+}$.

Show that $v_{i}\in \mathcal{W}^{s,p\left( .,.\right) }\left( Q_{+}\right) $%
, by setting $x=$ $H_{i}\left( \widehat{x}\right) $, we have:%
\begin{eqnarray*}
&&\int_{Q_{+}}\int_{Q_{+}}\frac{\left\vert v\left( \widehat{x}\right)
-v\left( \widehat{y}\right) \right\vert ^{p\left( H_{i}\left( \widehat{x}%
\right) ,H_{i}\left( \widehat{y}\right) \right) }}{\lambda ^{p\left(
H_{i}\left( \widehat{x}\right) ,H_{i}\left( \widehat{y}\right) \right)
}\left\vert \widehat{x}-\widehat{y}\right\vert ^{n+sp\left( H_{i}\left( 
\widehat{x}\right) ,H_{i}\left( \widehat{y}\right) \right) }}d\widehat{x}d%
\widehat{y} \\
&=&\int_{Q_{+}}\int_{Q_{+}}\frac{\left\vert u\left( H_{i}\left( \widehat{x}%
\right) \right) -u\left( H_{i}\left( \widehat{y}\right) \right) \right\vert
^{p\left( H_{i}\left( \widehat{x}\right) ,H_{i}\left( \widehat{y}\right)
\right) }}{\lambda ^{p\left( H_{i}\left( \widehat{x}\right) ,H_{i}\left( 
\widehat{y}\right) \right) }\left\vert \widehat{x}-\widehat{y}\right\vert
^{n+sp\left( H_{i}\left( \widehat{x}\right) ,H_{i}\left( \widehat{y}\right)
\right) }}d\widehat{x}d\widehat{y} \\
&=&\int_{B_{i}\cap \Omega }\int_{B_{i}\cap \Omega }\frac{\left\vert u\left(
x\right) -u\left( y\right) \right\vert ^{p\left( x,y\right) }}{\lambda
^{p\left( x,y\right) }\left\vert H_{i}^{-1}\left( x\right) -H_{i}^{-1}\left(
y\right) \right\vert ^{n+sp\left( x,y\right) }}\det \left( H_{i}^{-1}\right)
dxdy
\end{eqnarray*}%
since $\det \left( H_{i}^{-1}\right) \in L^{\infty }\left( B_{i}\cap \Omega
\right) $ and $H_{i}$ is a bi-lipschitz map, then: 
\begin{eqnarray*}
&&\int_{Q_{+}}\int_{Q_{+}}\frac{\left\vert v\left( \widehat{x}\right)
-v\left( \widehat{y}\right) \right\vert ^{p\left( H_{i}\left( \widehat{x}%
\right) ,H_{i}\left( \widehat{y}\right) \right) }}{\lambda ^{p\left(
H_{i}\left( \widehat{x}\right) ,H_{i}\left( \widehat{y}\right) \right)
}\left\vert \widehat{x}-\widehat{y}\right\vert ^{n+sp\left( H_{i}\left( 
\widehat{x}\right) ,H_{i}\left( \widehat{y}\right) \right) }}d\widehat{x}d%
\widehat{y} \\
&\leq &C\int_{B_{i}\cap \Omega }\int_{B_{i}\cap \Omega }\frac{\left\vert
u\left( x\right) -u\left( y\right) \right\vert ^{p\left( x,y\right) }}{%
\lambda ^{p\left( x,y\right) }\left\vert x-y\right\vert ^{n+sp\left(
x,y\right) }}dxdy\text{.}
\end{eqnarray*}%
Note that this integral is finite since $u\in W^{s,p\left( .,.\right)
}\left( B_{i}\cap \Omega \right) $. Now by lemma \ref{extension par
reflexion} we can extend $v_{i}$ to all $Q$ so that the extension $%
v_{i}^{\ast }\in W^{s,p\left( .,.\right) }\left( Q\right) $. Retransfer $%
v_{i}^{\ast }$ to $B_{i}$ using $H_{i}^{-1}$ and we set: 
\begin{equation*}
w_{i}\left( x\right) :=v_{i}^{\ast }\left( H_{i}^{-1}\left( x\right) \right) 
\text{, for any }x\in B_{i}\text{.}
\end{equation*}%
Since $H_{i}$ is a bi-lipschitz map, by arguing as above it follows that $%
w_{i}\in \mathcal{W}^{s,p\left( .,.\right) }\left( B_{i}\right) $, $w_{i}=u$
on $B_{i}\cap \Omega $, and consequently $\psi _{i}w_{i}=\psi _{i}u$ on $%
B_{i}\cap \Omega $. By definition $\psi _{i}w_{i}$ has compact support in $%
B_{i}$ and therefore, as done for $\psi _{0}u$, we can consider the
extension $\widetilde{\psi _{i}w_{i}}$ to all $%
\mathbb{R}
^{n}$ a way that $\widetilde{\psi _{i}w_{i}}\in \mathcal{W}^{s,p\left(
.,.\right) }\left( 
\mathbb{R}
^{n}\right) $, we set for $x\in 
\mathbb{R}
^{n}$, 
\begin{equation*}
\widetilde{\psi _{i}w_{i}}\left( x\right) =\left\{ 
\begin{array}{c}
\psi _{i}w_{i}\left( x\right) \text{ \ if\ \ \ \ }x\in B_{i}\text{,} \\ 
0\text{ \ \ \ \ if\ \ \ \ }x\in 
\mathbb{R}
^{n}\backslash B_{i}\text{\ .\ \ \ \ \ \ \ \ }%
\end{array}%
\right.
\end{equation*}%
Finally, let the operator:%
\begin{equation*}
\mathcal{E}u=\widetilde{\psi _{0}u}+\sum_{i=1}^{k}\widetilde{\psi _{i}w_{i}}
\end{equation*}%
be the extension of $u$ defined on all $%
\mathbb{R}
^{n}$. By construction we have $\mathcal{E}u\mid \Omega =u$ and $\mathcal{E}$
is a linear extension operator. On the other hand%
\begin{eqnarray*}
\left\Vert \mathcal{E}u\right\Vert _{\mathcal{W}^{s,p\left( .,.\right)
}\left( 
\mathbb{R}
^{n}\right) } &=&\left\Vert \widetilde{\psi _{0}u}+\sum_{i=1}^{k}\widetilde{%
\psi _{i}w_{i}}\right\Vert _{_{\mathcal{W}^{s,p\left( .,.\right) }\left( 
\mathbb{R}
^{n}\right) }} \\
&\leq &\left\Vert \widetilde{\psi _{0}u}\right\Vert _{\mathcal{W}^{s,p\left(
.,.\right) }\left( 
\mathbb{R}
^{n}\right) }+\sum_{i=1}^{k}\left\Vert \widetilde{\psi _{i}w_{i}}\right\Vert
_{\mathcal{W}^{s,p\left( .,.\right) }\left( 
\mathbb{R}
^{n}\right) }\text{.}
\end{eqnarray*}%
If $\left\Vert u\right\Vert _{\mathcal{W}^{s,p\left( .,.\right) }\left(
\Omega \right) }\geq 1$, then 
\begin{eqnarray*}
\left\Vert \mathcal{E}u\right\Vert _{\mathcal{W}^{s,p\left( .,.\right)
}\left( 
\mathbb{R}
^{n}\right) } &\leq &C\left\Vert u\right\Vert _{\mathcal{W}^{s,p\left(
.,.\right) }\left( \Omega \right) }^{\beta }+\sum_{i=1}^{k}C_{i}\left\Vert
\psi _{i}w_{i}\right\Vert _{\mathcal{W}^{s,p\left( .,.\right) }\left(
B_{i}\right) }^{\beta } \\
&\leq &C\left\Vert u\right\Vert _{\mathcal{W}^{s,p\left( .,.\right) }\left(
\Omega \right) }^{\beta }+\sum_{i=1}^{k}C_{i}\left\Vert w_{i}\right\Vert _{%
\mathcal{W}^{s,p\left( .,.\right) }\left( B_{i}\right) }^{\beta } \\
&\leq &C\left\Vert u\right\Vert _{\mathcal{W}^{s,p\left( .,.\right) }\left(
\Omega \right) }^{\beta }+\sum_{i=1}^{k}C_{i}\left\Vert w_{i}\right\Vert _{%
\mathcal{W}^{s,p\left( .,.\right) }\left( B_{i}\right) }^{\beta } \\
&\leq &C\left\Vert u\right\Vert _{\mathcal{W}^{s,p\left( .,.\right) }\left(
\Omega \right) }^{\beta }+\sum_{i=1}^{k}C_{i}\left\Vert v_{i}^{\ast
}\right\Vert _{\mathcal{W}^{s,p\left( .,.\right) }\left( Q\right) }^{\beta }
\\
&\leq &C\left\Vert u\right\Vert _{\mathcal{W}^{s,p\left( .,.\right) }\left(
\Omega \right) }^{\beta }+\sum_{i=1}^{k}C_{i}\left\Vert v_{i}\right\Vert _{%
\mathcal{W}^{s,p\left( .,.\right) }\left( Q^{+}\right) }^{\beta } \\
&\leq &C\left\Vert u\right\Vert _{\mathcal{W}^{s,p\left( .,.\right) }\left(
\Omega \right) }^{\beta }+\sum_{i=1}^{k}C_{i}\left\Vert u\right\Vert _{%
\mathcal{W}^{s,p\left( .,.\right) }\left( \Omega \cap B_{i}\right) }^{\beta }
\\
&\leq &\mathcal{C}\left\Vert u\right\Vert _{W^{s,p\left( .,.\right) }\left(
\Omega \right) }^{\beta }\text{.}
\end{eqnarray*}%
By the same way we get 
\begin{equation*}
\left\Vert \mathcal{E}u\right\Vert _{\mathcal{W}^{s,p\left( .,.\right)
}\left( 
\mathbb{R}
^{n}\right) }\leq C\left\Vert u\right\Vert _{\mathcal{W}^{s,p\left(
.,.\right) }\left( \Omega \right) }^{\alpha }
\end{equation*}%
if $\left\Vert u\right\Vert _{\mathcal{W}^{s,p\left( .,.\right) }\left(
\Omega \right) }\leq 1$. Finally the operator:%
\begin{equation*}
\mathcal{E}u=\widetilde{\psi _{0}u}+\sum_{i=1}^{k}\widetilde{\psi _{i}w_{i}}
\end{equation*}%
possesses all the desired properties.\qed
\end{proof}

\section{Complemented subspaces in $\mathcal{W}^{s,p\left( .,.\right)
}\left( 
\mathbb{R}
^{n}\right) $}

\bigskip The complemented subspace problem plays a key role in the
development of the Banach space theory. In the case of Hilbert space $H$ it
is know that every closed subspace $Y\subset H$ is complemented; the
orthogonal complement $Y^{\bot }$ is a closed subspace of $H$ and we have 
\begin{equation*}
H=Y\oplus Y^{\bot }\text{.}
\end{equation*}

The lack of the Hilbert structure of the space $\mathcal{W}^{s,p\left(
.,.\right) }\left( 
\mathbb{R}
^{n}\right) $ make the complemented subspace problem in this space very
difficult. In this section we will study this problem by using the previous
extension theorem.

The trace of a function is in some sense a restriction of the function to a
subset of the original set of definition.

\begin{definition}
For any $u\in \mathcal{W}^{s,p\left( .,.\right) }\left( 
\mathbb{R}
^{n}\right) $, define the trace operator $\mathcal{T}$ by 
\begin{equation*}
\begin{array}{ccc}
\mathcal{T}:\mathcal{W}^{s,p\left( .,.\right) }\left( 
\mathbb{R}
^{n}\right) & \rightarrow & \mathcal{W}^{s,p\left( .,.\right) }\left( \Omega
\right)%
\end{array}%
\text{, }\mathcal{T}u=u\mid _{\Omega }\text{.}
\end{equation*}
\end{definition}

Note that if $\mathcal{E}$ is an extension operator, then $\mathcal{T}\circ 
\mathcal{E}$ is the identity on $\mathcal{W}^{s,p\left( .,.\right) }\left(
\Omega \right) $.

\begin{corollary}
The operator $\mathcal{T}$ is surjective.
\end{corollary}

\begin{proof}
by theorem \ref{extension operator} every function $u$ in $\mathcal{W}%
^{s,p\left( .,.\right) }\left( \Omega \right) $ admits an extension to $%
\mathcal{W}^{s,p\left( .,.\right) }\left( 
\mathbb{R}
^{n}\right) $ then the trace operator $\mathcal{T}$ is surjective.\qed
\end{proof}

\begin{theorem}
\label{Theorem of supplementarity}Suppose that $\Omega $ is of classe $%
\mathcal{C}^{1}$ with $\partial \Omega $ bounded. Then the subspace $Ker%
\mathcal{T}$ is complemented in $\mathcal{W}^{s,p\left( .,.\right) }\left( 
\mathbb{R}
^{n}\right) $.
\end{theorem}

\begin{proof}
By theorem \ref{extension operator} there exist a linear extension operator 
\begin{equation*}
\begin{array}{ccc}
\mathcal{E}: & \mathcal{W}^{s,p\left( .,.\right) }\left( \Omega \right)
\rightarrow & \mathcal{W}^{s,p\left( .,.\right) }\left( 
\mathbb{R}
^{n}\right)%
\end{array}%
\end{equation*}%
such that 
\begin{equation*}
\forall x\in \Omega :\mathcal{E}u\left( x\right) =u\left( x\right) \text{.}
\end{equation*}%
We have 
\begin{equation*}
\mathcal{E}\left( \mathcal{W}^{s,p\left( .,.\right) }\left( \Omega \right)
\right) \subset \mathcal{W}^{s,p\left( .,.\right) }\left( 
\mathbb{R}
^{n}\right)
\end{equation*}
is a closed subspace, and $\ $%
\begin{equation*}
\forall u\in \mathcal{W}^{s,p\left( .,.\right) }\left( 
\mathbb{R}
^{n}\right) :u=u-\mathcal{E}\left( \mathcal{T}u\right) +\mathcal{E}\left( 
\mathcal{T}u\right) .
\end{equation*}

Since 
\begin{equation*}
\left( u-\mathcal{E}\left( \mathcal{T}u\right) \right) \in Ker\mathcal{T}%
\text{, \ }\mathcal{E}\left( \mathcal{T}u\right) \in \mathcal{E}\left( 
\mathcal{W}^{s,p\left( .,.\right) }\left( \Omega \right) \right) \text{,}
\end{equation*}
and 
\begin{equation*}
Ker\mathcal{T\cap E}\left( \mathcal{W}^{s,p\left( .,.\right) }\left( \Omega
\right) \right) =\left\{ 0\right\} ,
\end{equation*}
we conclude that 
\begin{equation*}
\mathcal{W}^{s,p\left( .,.\right) }\left( 
\mathbb{R}
^{n}\right) =Ker\mathcal{T}\oplus \mathcal{E}\left( \mathcal{W}^{s,p\left(
.,.\right) }\left( \Omega \right) \right) \text{.}
\end{equation*}%
and this complete the proof.\qed
\end{proof}

\begin{corollary}
If $p\left( x,y\right) =2$, then :%
\begin{equation*}
\mathcal{W}^{s,2}\left( 
\mathbb{R}
^{n}\right) =H^{s}\left( 
\mathbb{R}
^{n}\right) =Ker\mathcal{T}\oplus \left( Ker\mathcal{T}\right) ^{\perp }%
\text{.}
\end{equation*}
\end{corollary}

\begin{proof}
In this case, theorem \ref{Theorem of supplementarity} is a direct
consequence of the Hilbert structure of the Space $\mathcal{W}^{s,2}\left( 
\mathbb{R}
^{n}\right) $. Indeed $Ker\mathcal{T}$ is a closed subspace of $\mathcal{W}%
^{s,2}\left( 
\mathbb{R}
^{n}\right) $, so 
\begin{equation*}
\mathcal{W}^{s,2}\left( 
\mathbb{R}
^{n}\right) =Ker\mathcal{T}\oplus \left( Ker\mathcal{T}\right) ^{\perp }%
\text{.}
\end{equation*}\qed
\end{proof}




\end{document}